\documentclass[12pt]{amsart}

\usepackage[margin=1.15in]{geometry}
\usepackage{amscd,amssymb, amsmath, wasysym, mathrsfs, mathtools}
\usepackage{graphicx}
\usepackage[all, cmtip]{xy}

\usepackage{url}
\usepackage[pagebackref=true]{hyperref}

\theoremstyle{plain}
\newtheorem{theorem}{Theorem}[section]
\newtheorem{prop}[theorem]{Proposition}

\newtheorem*{claim}{Claim}

\newtheorem{lemma}[theorem]{Lemma}

\theoremstyle{definition}

\newtheorem{rmk}[theorem]{Remark}

\newtheorem*{ex*}{Example}

\def\lam{\lambda}

\newcommand\sO{{\mathcal O}}

\newcommand\sD{\mathcal{D}}

\def\gg{{\mathbb{G}}}
\newcommand\qq{{\mathbb{Q}}}
\newcommand\zz{{\mathbb{Z}}}

\newcommand\cc{{\mathbb{C}}}
\newcommand\aaa{{\mathbb{A}}}

\newcommand\hh{{\mathbb{H}}}

\newcommand\mb{\mathbf{M}_{\textup{B}}}
\newcommand\mdr{\mathbf{M}_{\textup{DR}}}

\DeclareMathOperator{\codim}{codim}              

\DeclareMathOperator{\homo}{Hom}

\DeclareMathOperator{\spec}{Spec}

\title{Cohomology jump loci of quasi-projective varieties}

\begin{document}
\author{Nero Budur}
\email{Nero.Budur@wis.kuleuven.be}
\address{KU Leuven and University of Notre Dame}
\curraddr{KU Leuven, Department of Mathematics,
Celestijnenlaan 200B, B-3001 Leuven, Belgium}

\author{Botong Wang}
\email{bwang3@nd.edu}
\address{University of Notre Dame}
\curraddr{Department of Mathematics,
 255 Hurley Hall, IN 46556, USA}



\begin{abstract} 
We prove that the cohomology jump loci in the space of rank one local systems over a smooth quasi-projective variety are finite unions of torsion translates of subtori. The main ingredients are a recent result of Dimca-Papadima, some techniques introduced by Simpson, together with properties of the moduli space of logarithmic connections constructed by Nitsure and Simpson.
\end{abstract}

\maketitle

\section{Introduction}
Let $X$ be a connected, finite-type CW-complex. Define $$\mb(X)=\homo(\pi_1(X), \cc^*)$$ to be the variety of $\cc^*$ representations of $\pi_1(X)$. Then $\mb(X)$ is a direct product of $(\cc^*)^{b_1(X)}$ and a finite abelian group. For each point $\rho\in \mb(X)$, there exists a unique rank one local system $L_\rho$, whose monodromy representation is $\rho$. The {\it cohomology jump loci} of $X$ are the natural strata $$\Sigma^i_k(X)=\{\rho\in \mb(X)\; |\; \dim_\cc H^i(X, L_\rho)\geq k\}.$$  $\Sigma^i_k(X)$ is a Zariski closed subset of $\mb(X)$. A celebrated result of Simpson says that if $X$ is a smooth projective variety defined over $\cc$, then $\Sigma^i_k(X)$ is a union of torsion translates of subtori of $\mb(X)$.

In this paper, we generalize Simpson's result to quasi-projective varieties. 
\begin{theorem}\label{torsion}
Suppose $U$ is a smooth quasi-projective variety defined over $\cc$.  Then $\Sigma^i_k(U)$ is a finite union of torsion translates of subtori of $\mb(U)$. 
\end{theorem}

When $U$ is compact, the theorem is proved in \cite{gl1, gl2}, \cite{a1}, \cite{s2}, with the strongest form appearing in the latter. When $b_1(\bar{U})=0$, Arapura \cite{a2} showed that $\Sigma^i_k(U)$ are union of translates of subtori. The case of unitary rank one local systems on $U$ has been considered in \cite{b} and \cite{l}. Libgober \cite{l} also proved the same theorem for $U=\mathcal{X}-\mathcal{D}$ where $\mathcal{X}$ is a germ of a smooth analytic space, and $\mathcal{D}$ is a divisor of $\mathcal{X}$. Dimca and Papadima were able to prove the following:

\begin{theorem}\cite[Theorem C]{dp}\label{dimca}
Under the same assumption as Theorem \ref{torsion}, every irreducible component of $\Sigma^i_k(U)$ containing $\mathbf{1}\in \mb(U)$ is a subtorus. 
\end{theorem}
The proof of this result reduces to the study of the infinitesimal deformations with cohomology constraints of the trivial local system. These are governed in general by infinite-dimensional models. In \cite{dp} it is shown that, in this case, the finite-dimensional Gysin model due to Morgan provides the necessary linear algebra description for the infinitesimal deformations.

The result of Dimca and Papadima serves as a key ingredient of our theorem. In Section 2, we will show that each irreducible component of $\Sigma^i_k(U)$ contains a torsion point. Then, in Section 3, we will see that, thanks to Theorem \ref{dimca}, having a torsion point on an irreducible component of $\Sigma^i_k(U)$ forces this component to be a translate of subtorus. 

There are two other proofs of Simpson's theorem: one via positive characteristic methods \cite{PR}, and one via D-modules \cite{Sc1, Sc2}. However, in this paper we follow the original approach of Simpson. There are no analogous results for higher rank local systems even in the projective case.

\medskip
\noindent
{\bf Acknowledgement.} The first author was partially supported by  the NSA, the Simons Foundation grant 245850, and the BOF-OT KU Leuven grant. 

\section{Torsion points on the cohomology jump loci}
Let $X$ be a smooth complex projective variety, and let $D=\sum_{\lambda=1}^n D_\lambda$ be a simple normal crossing divisor on $X$ with irreducible components $D_\lam$. Let $U=X-D$. Thanks to Hironaka's theorem on resolution of singularities, every smooth quasi-projective variety $U$ can be realized in this way. The goal of this section is to prove the following:
\begin{theorem}\label{torsionp}
Each irreducible component of $\Sigma^i_k(U)$ contains a torsion point. 
\end{theorem}
First, we want to reduce to the case when $X$ and each $D_\lam$ are defined over $\bar\qq$. This can be done using a technique which we have learnt from the proof of \cite[Theorem 4.1]{s2}. We reproduce it here. 

We can assume $X$ and each $D_\lambda$ to be defined over a subring $O$ of $\cc$, which is finitely generated over $\qq$. Denote the embedding of $O$ to $\cc$ by $\sigma: O\to \cc$. Each ring homomorphism $O\to \cc$ corresponds to a point in $\spec(O)(\cc)$. Denote by $X^0$ and $D_\lambda^0$ the schemes over $\spec(O)$ which give rise to $X$ and $D_\lambda$ respectively after tensoring with $\cc$, that is $X=X^0\times_{\spec(O)}\spec(\cc)$ and $D_\lambda=D_\lambda^0\times_{\spec(O)}\spec(\cc)$. By possibly replacing $O$ by $O[\frac{1}{h}]$ for some $h\in O$, we can assume $X^0$ and every $D_\lambda^0$ are smooth over $\spec(O)$, and all the intersections of $D_\lambda^0$'s are transverse. Since each connected component of $\spec(O)(\cc)$ contains a $\bar\qq$ point, there exists a point $P\in \spec(O)(\bar\qq)$, and a continuous path from $\sigma\in \spec(O)(\cc)$ to $P$ in $\spec(O)(\cc)^{\textrm{top}}$. Then, according to Thom's First Isotopy Lemma \cite[Ch. 1, Theorem 3.5]{db}, $X^0(\cc)$ together with its strata given by the $D_\lambda^0(\cc)$, is a topologically locally trivial fibration in the stratified sense over $\spec(O)(\cc)^{\textrm{top}}$. In particular, letting $X'$ and $D'_\lambda$ be the corresponding fibers over $P$, transporting along the path gives an isomorphism $(X-D)^{\textrm{top}}\cong (X'-D')^{\textrm{top}}$. Recall  that $\mb(U)$ and $\Sigma^i_k(U)$ depend only on the topology of $U$. Hence replacing $U=X-D$ by $U'=X'-D'$, we may assume that $X$ and each $D_\lam$ are defined over $\bar\qq$. 

Next, we introduce the other side of the story, namely the logarithmic flat bundles on $(X, D)$. A logarithmic flat bundle on $(X, D)$ consists of a vector bundle $E$ on $X$, and a logarithmic connection $\nabla: E\to E\otimes \Omega^1_X(\log D)$, satisfying the integrability condition $\nabla^2=0$. Given a logarithmic flat bundle $(E, \nabla)$, the flat sections of $E$ on $U$ (by which we will always mean on $U^{\textrm{top}}$) form a local system. And conversely, given any local system $L$ on $U$ (by which, as in the introduction, we will always mean a local system on $U^{\textrm{top}}$), it is always obtained from some logarithmic flat bundle $(E, \nabla)$. However, different logarithmic flat bundles may give the same local system. This correspondence between local systems on $U$ and logarithmic flat bundles on $(X, D)$ is very well understood (e.g. \cite{d}, \cite{s1}, \cite{m}). 

For a vector bundle $E$ on $X$, the structure of a logarithmic flat bundle $(E,\nabla)$ on $(X,D)$ is the same as a $\sD_X(\log D)$-module structure on $E$, where $\sD_X(\log D)$ is the sheaf of logarithmic differentials.

Nitsure \cite{n} and Simpson \cite{s3} constructed coarse moduli spaces, which are separated quasi-projective schemes, for Jordan-equivalence classes of semistable $\Lambda$-modules which are $\sO_X$-coherent and torsion free, where $\Lambda$ is a sheaf of rings of differential operators. The two examples of $\Lambda$ which we are concerned with are $\sD_X$, the usual sheaf of differential operators on $X$, and $\sD_X(\log D)$, the sheaf of logarithmic differentials. We denote by $\mdr(X)$ and $\mdr(X/D)$ the moduli space of rank one $\sD_X$-modules and the moduli space of rank one $\sD_X(\log D)$-modules, respectively. In the rank one case, semistable is the same as stable and this condition is automatic as is the locally free condition, and Jordan-equivalence is the same as isomorphic. Thus, the points of $\mdr(X)$ and $\mdr(X/D)$ correspond to isomorphism classes of flat, respectively, logarithmic flat line bundles. Since we did not put any condition on the Chern class of the underlying line bundles, in general $\mdr(X/D)$ has infinitely many connected components. $\mdr(X)$, $\mdr(X/D)$, $\mb(X)$ and $\mb(U)$ are all algebraic groups, except $\mdr(X/D)$ may not be of finite type. 

The following diagram plays an essential role in our proof.
$$
\xymatrix{
&&0\ar[d]&0\ar[d]\\
&&\zz^n\ar[d]\ar@2{-}[r]&\zz^n\ar[d]\\
0\ar[r]&\mdr(X)\ar[d]^{RH}\ar[r]&\mdr(X/D)\ar[d]^{RH}\ar[r]^{\quad res}&\cc^n\ar[d]^\exp \\
0\ar[r]&\mb(X)\ar[r]&\mb(U)\ar[d]\ar[r]^{ev}&(\cc^*)^n\ar[d]\\
&&0&0
}
$$

Let us first explain how the arrows are defined. Since every $\sD_X$-module is naturally a $\sD_X(\log D)$-module, there is a natural embedding $\mdr(X)\hookrightarrow \mdr(X/D)$. On the other hand, the embedding $U\hookrightarrow X$ induces a surjective map on the fundamental group $\pi_1(U)\to \pi_1(X)$. Composing this map with the representations, we have $\mb(X)\hookrightarrow \mb(U)$. For every rank one logarithmic flat bundle $(E, \nabla)$, taking the residue along each $D_\lambda$ is the map $res$. In other words, $res((E, \nabla))=\{res_{D_\lambda}(\nabla)\}_{1\leq \lambda\leq n}$. Around each $D_\lambda$, we can take a small loop $\gamma_\lambda$. The map $ev$ is the evaluation at the loops $\gamma_\lambda$. More precisely $ev(\rho)=\{\rho(\gamma_\lambda)\}_{1\leq \lambda\leq n}$. 

For the horizontal arrows, $RH: \mdr(X)\to \mb(X)$ is taking the monodromy representations for flat bundles. Since every logarithmic flat bundle on $(X, D)$ restricts to a flat bundle on $U$, taking the monodromy representation on $U$ is $RH: \mdr(X/D)\to \mb(U)$. The map $\exp: \cc^n\to (\cc^*)^n$ is component-wise defined to be multiplying by $2\pi\sqrt{-1}$, then taking exponential. On $\mdr(X/D)$, there are some special elements. Let $(\sO_X, d)$ be the trivial rank one logarithmic flat bundle on $(X, D)$. Notice that $\sO_X(-D_\lambda)$ is preserved under $d$, that is, there is an induced map $d: \sO_X(-D_\lambda)\to \sO_X(-D_\lambda)\otimes \Omega^i_X(\log D)$. Therefore, $(\sO_X(-D_\lambda), d)$ is also a logarithmic flat bundle on $(X, D)$. The map $\zz^n\to \mdr(X/D)$ is defined by $\{m_\lambda\}_{1\leq \lambda\leq n}\mapsto \bigotimes_{1\leq \lambda\leq n}(\sO_X(-D_\lambda), d)^{\otimes m_\lambda}$. The map $\zz^n\to \cc^n$ is the natural inclusion map. 

Notice that all the maps are group homomorphisms, all the rows and columns are exact.  The first map $RH$ is an analytic isomorphism, since $\mdr(X)$ and $\mb(X)$ analytically represent the same functor (\cite{s4}). Similarly, the quotient $\mdr(X/D)/\zz^n$ and $\mb(U)$ represent the same functor from the category of analytic spaces to the category of sets. Therefore, by Yoneda's lemma, $RH: \mdr(X/D)\to \mb(U)$ is an analytic covering map with transformation group $\zz^n$. The map $\exp$ is obviously an analytic covering map. 

According to the discussion following Theorem \ref{torsionp}, we can assume $X$ and each $D_\lam$ to be defined over $\bar\qq$ without loss of generality. Then $\mdr(X/D)$ and $\mdr(X)$ are also defined over $\bar\qq$. The representation varieties $\mb(U)$ and $\mb(X)$ are always defined over $\qq$. Therefore, all the horizontal arrows in the above diagram are maps defined over $\bar\qq$. From now on, we should think of $\cc^n$ and $(\cc^*)^n$ as varieties defined over $\bar\qq$, or in other words, as $\aaa^n_\cc=\aaa_{\bar\qq}^n\times_{\bar\qq}\cc$ and $(\gg_{m,\cc})^n=(\gg_{m,{\bar\qq}})^n\times_{\bar\qq}\cc$, respectively.

\begin{lemma}\label{cstar}
Suppose $Z\subset \cc^n$ is a non-empty Zariski constructible subset defined over $\bar\qq$. Suppose $\exp(Z)\subset (\cc^*)^n$ is also a a Zariski constructible subset defined over $\bar\qq$. Then $\exp(Z)$ contains a torsion point. 
\end{lemma}
\begin{proof}
When $n=1$, this follows from the Gelfond-Schneider theorem, which says if $\alpha$ and $e^{2\pi\sqrt{-1}\alpha}$ are both algebraic numbers, then $\alpha\in \qq$. 

We use induction on $n$. Suppose the lemma is true for $\cc^{n-1}$. Let $p_1: \cc^n\to \cc^{n-1}$ and $p_2: (\cc^*)^n\to (\cc^*)^{n-1}$ be the projections to the first $n-1$ factors. Then $p_1(Z)\subset \cc^{n-1}$ and $p_2(\exp(Z))\subset (\cc^*)^{n-1}$ are both defined over $\bar\qq$. Since $\exp(p_1(Z))=p_2(\exp(Z))$, by induction hypothesis, $\exp(p_1(Z))$ contains a torsion point $\tau$ in $(\cc^*)^{n-1}$. 

Let $M=p_2^{-1}(\tau)$. Then $\exp^{-1}(M)$ is a disjoint union of infinitely many copies of $\cc$. Choose one copy of those, which intersects with $Z$. Denote this copy by $N$. Since $\tau$ is a torsion point in $(\cc^*)^{n-1}$, $N$ is defined by equations with $\qq$ coefficients. Consider the following diagram,
$$
\xymatrix{
N\ar[r]^{q_1}\ar[d]_{\exp}&\cc\ar[d]^{\exp}\\
M\ar[r]^{q_2}&\cc^*
}
$$
where $q_1$ and $q_2$ are projections to the last coordinates respectively. Then $q_1$ and $q_2$ are isomorphisms defined over $\bar\qq$. If $M\subset \exp(Z)$, then every torsion point in $\cc^*$ via $q_2^{-1}$ gives a torsion point in $\exp(Z)$. If $M\nsubseteq \exp(Z)$, then $M\cap \exp(Z)$ contains finitely many points. Hence, $N\cap Z$ also contains finitely many points. In this case, let $\sigma$ be any point in $N\cap Z$, $q_1(\sigma)\in \cc$ is defined over $\bar\qq$. On the other hand,  $\exp(\sigma)$ is a point in $M\cap \exp(Z)$, and hence defined over $\bar\qq$. Thus, $q_2(\exp(\sigma))=\exp(q_1(\sigma))\in \cc^*$ is defined over $\bar\qq$. Now, the Gelfond-Schneider theorem implies that $p_2(\exp(\sigma))$ is torsion in $\cc^*$. Since $q_2(\exp(\sigma))$ is torsion in $\cc^*$ and $p_2(\exp(\sigma))=\tau$ is torsion in $(\cc^*)^{n-1}$, $\exp(\sigma)\in \exp(Z)$ is a torsion point in $(\cc^*)^n$.  
\end{proof}
\begin{rmk}
In fact, Jiu-Kang Yu has pointed out to us that, using Hilbert's irreducibility theorem one, can prove that if $Z$ and $\exp(Z)$ are closed irreducible subvarieties defined over $\bar\qq$, then $\exp(Z)$ is a torsion translate of subtorus. We give the proof in the appendix. 
\end{rmk}

Remember that we assume that $X$ and each $D_\lam$ are defined over $\bar\qq$. 

\begin{lemma}\label{tas}
Let $T$ be an irreducible component of $\Sigma^i_k(U)$. Then there exists an irreducible subvariety $S$ of $\mdr(X/D)$ defined over $\bar\qq$ such that $RH(S)=T$. 
\end{lemma}
\begin{proof}
For any $\rho\in \mb(U)$, $RH^{-1}(\rho)$ contains all the possible extensions of $L_\rho$ to a logarithmic flat bundle over $(X, D)$. Suppose $(E, \nabla)\in RH^{-1}(L)$, and suppose $\nabla$ does not have any residue being equal to a positive integer, that is, $res((E, \nabla))$ does not have any positive integer in its coordinates. Then by a theorem of Deligne \cite[II, 6.10]{d}, the hypercohomology of the algebraic de Rham complex 
$$E\otimes \Omega^\bullet_X(\log D)=[E\stackrel{\nabla}{\longrightarrow}E\otimes \Omega^1_X(\log D)\stackrel{\nabla}{\longrightarrow}E\otimes \Omega^2_X(\log D)\stackrel{\nabla}{\longrightarrow}\cdots]$$
computes the cohomology of the local system $L$, i.e., $\hh^i(X, E\otimes \Omega^\bullet_X(\log D))\cong H^i(U, L_\rho)$.

Define the bad locus ${BL}\subset \mdr(X/D)$ to be the locus where one of the residues of $\nabla$ is a positive integer. Then ${BL}$ is the preimage of infinitely many hyperplanes in $\cc^n$ via $res$. Define $$\Sigma^i_k(X/D)=\{(E, \nabla)\in \mdr(X/D)\;|\; \dim \hh^i(X, E\otimes \Omega^\bullet_X(\log D))\geq k\}.$$ 
Given any point $\rho_0$ in $\Sigma^i_k(U)$, one can always find an extension $(E_0, \nabla_0)\in \mdr(X/D)$ of $L_{\rho_0}$, which is not in ${BL}$, e.g., the Deligne extension. Then  $RH(\Sigma^i_k(X/D)-{BL})=\Sigma^i_k(U)$. 

Now, given $T\subset \Sigma^i_k(U)$ as an irreducible component, take any point $\rho_0$ in $T$. Since $RH$ is analytically a covering map, there is a unique irreducible component $S$ of $RH^{-1}(T)$ containing the Deligne extension $(E_0, \nabla_0)$ of $L_{\rho_0}$. Since $S\nsubseteq {BL}$ and since $RH$ is analytically a covering map,  we have $RH(S)=T$. By semicontinuity theorem, $\Sigma^i_k(X/D)\subset \mdr(X/D)$ is closed and defined over $\bar\qq$. Since $S$ is an irreducible component of $\Sigma^i_k(X/D)$, $S$ is closed and defined over $\bar\qq$. \end{proof}

Now, we are ready to prove Theorem \ref{torsionp}. 
\begin{proof}[Proof of Theorem \ref{torsionp}]
Let $T$ be an irreducible component of $\Sigma^i_k(U)$. By \cite[Lemma 9.2]{dp}, $\Sigma^i_k(U)$ is defined by some Fitting ideal coming from the CW-complex structure of $U$. Thus, $\Sigma^i_k(U)\subset \mb(U)$ is defined over $\qq$. Hence, as an irreducible component of $\Sigma^i_k(U)$, $T$ is defined over $\bar\qq$. According to Lemma \ref{tas}, there exists $S\subset\mdr(X/D)$ defined over $\bar\qq$ such that $RH(S)=T$. Then $res(S)\subset \cc^n$ and $ev(T)\subset (\cc^*)^n$ are defined over $\bar\qq$, and moreover, $\exp(res(S))=ev(T)$. According to Lemma \ref{cstar}, $ev(T)$ contains a torsion point $\tau$. 

Since $\tau\in (\cc^*)^n$ is torsion, we can take $l\in \zz_+$ such that $\tau^l=\mathbf{1}\in (\cc^*)^n$. Then the image of the $l$-power map $(\,\cdot\,)^l:ev^{-1}(\tau){\longrightarrow} \mb(U)$ is equal to $\mb(X)$.  Choose $\eta\in ev^{-1}(\tau)$, such that $\eta^l=\mathbf{1}$. 
Every $\xi\in RH^{-1}(\eta)$ is a $\bar\qq$ point in $\mdr(X/D)$. In fact, since $\eta^l=\mathbf{1}$ in $\mb(U)$, $\xi^l$ is in the image of $\zz^n\to \mdr(X/D)$. Recall that the image of $\{m_\alpha\}_{1\leq \alpha\leq n}$ under $\zz^n\to \mdr(X/D)$ is $\bigotimes_{1\leq \lambda\leq n}(\sO_X(-D_\lambda), d)^{\otimes m_\lambda}$, which is clearly a $\bar\qq$ point in $\mdr(X/D)$. Therefore, $\xi$, as an $l$-th root of a $\bar\qq$ point, has to be a $\bar\qq$ point. 

Notice that 
$$(ev\circ RH)^{-1}(\tau)=\bigcup_{\xi\in RH^{-1}(\eta)}\left(\xi\cdot\mdr(X)\right).$$
Moreover, since $RH(S)=T$, 
\begin{align*}
T\cap ev^{-1}(\tau)&=RH(S)\cap ev^{-1}(\tau)\\
&=RH\big(S\cap(RH\circ ev)^{-1}(\tau)\big)\\
&=\bigcup_{\xi\in RH^{-1}(\eta)}RH\Big(S\cap \big(\xi\cdot\mdr(X)\big)\Big).
\end{align*}
Each $RH(S\cap (\xi\cdot\mdr(X)))$ is closed in $\mb(U)$, and $T\cap ev^{-1}(\tau)$ is a noetherian topological space. Hence, for some $\xi_0\in RH^{-1}(\eta)$, $RH(S\cap (\xi_0\cdot\mdr(X)))$ contains an irreducible component of $T\cap ev^{-1}(\tau)$. Since $RH(\xi_0)=\eta$, $RH((\xi_0^{-1}\cdot S)\cap\mdr(X))$ contains an irreducible component of $\eta^{-1}\cdot(T\cap ev^{-1}(\tau))$. Recall that $\eta\in ev^{-1}(\tau)$. Thus, $\eta^{-1}\cdot(T\cap ev^{-1}(\tau))\subset \mb(X)$. 

Now, $RH$ maps an irreducible component $W$ of $(\xi_0^{-1}\cdot S)\cap\mdr(X)$ to an irreducible component $RH(W)$ of $\eta^{-1}\cdot(T\cap ev^{-1}(\tau))\subset \mb(X)$. Both of these irreducible components are defined over $\bar\qq$. Indeed, since $\xi_0$ and $S$ are defined over $\bar\qq$ in $\mdr(X/D)$, and $\eta$, $T$, $ev^{-1}(\tau)$ are defined over $\bar\qq$ in $\mb(U)$, $(\xi_0^{-1}\cdot S)\cap\mdr(X)$ and $\eta^{-1}\cdot(T\cap ev^{-1}(\tau))\subset \mb(X)$ are defined over $\bar\qq$ in $\mdr(X/D)$ and $\mb(U)$, respectively. Hence, the same is true for their irreducible components. Thus, we can apply \cite[Theorem 3.3]{s2} which says that this irreducible component $RH(W)$ of $\eta^{-1}\cdot(T\cap ev^{-1}(\tau))\subset \mb(X)$ is a torsion translate of a subtorus. In particular, $\eta^{-1}\cdot(T\cap ev^{-1}(\tau))\subset \mb(X)$ contains a torsion point. Since $\eta$ is also a torsion point, $T$ must contain a torsion point. 
\end{proof}

\section{Finite abelian covers}
First, we consider a more general situation. Let $U$ be a  connected, finite-type CW-complex, and let $\mb(U)=\homo (\pi_1(U), \cc^*)$ be the moduli space of rank one local systems on $U$, which is naturally an algebraic group. Suppose $\tau\in \mb(U)$ is a torsion point. Denote the universal cover of $U$ by $\tilde{U}$, and let $H$ be the kernel of $\tau: \pi_1(U)\to \cc^*$. Then $H$ acts on $\tilde{U}$ and we denote the quotient $\tilde{U}/H$ by $V$. Now, $\langle \tau\rangle$, the subgroup of $\mb(U)$ generated by $\tau$, acts on $V$, and the quotient $V/\langle \tau\rangle=U$. Denote this quotient by $f: V\to U$. Composing with $f_*: \pi_1(V)\to \pi_1(U)$, $f$ induces a homomorphism of algebraic groups $f^\star: \mb(U)\to \mb(V)$. Under this construction, we immediately have $f^\star(\tau)=\mathbf{1}\in\mb(V)$ is the identity element, i.e., $f^\star(\tau)$ maps every element in $\pi_1(V)$ to 1. The main result of this section is the following: 

\begin{prop}\label{cover}
Fixing $i$, suppose that for every $k\in \zz_+$, each irreducible component of $\Sigma^i_k(V)$ containing $\mathbf{1}$ is a subtorus. Then for every $k\in \zz_+$, each irreducible component of $\Sigma^i_k(U)$ containing $\tau$ is a translate of subtorus. 
\end{prop}
\begin{proof}
Denote the order of $\tau$ in $\mb(X)$ by $r$. For any local system $L$ on $U$, 
$$f_*f^*(L)\cong\bigoplus_{j=0}^{r-1}L\otimes_\cc L_\tau^{\otimes j}. $$
According to the projection formula, $H^i(V, f^*(L))\cong H^i(U, f_*f^*(L))$. Therefore, 
\begin{equation}\label{eq}
\dim H^i(V, f^*(L))=\sum\limits_{j=0}^{r-1}\dim H^i(U, L\otimes L_\tau^{\otimes j}).
\end{equation}

Let $T$ be an irreducible component of $\Sigma^i_k(U)$ containing $\tau$, and let $\rho$ be a general point in $T$. Define $\beta_j=\dim H^i(U, L_\rho\otimes L_\tau^{\otimes j})$, for $0\leq j\leq r-1$, and $\beta=\sum_{0\leq j\leq r-1} \beta_j$. It is possible that $T\subset\Sigma^i_{k+1}(U)$, and in this case, $\beta>k$. 
\begin{claim}
$f^\star(T)$ is an irreducible component of $\Sigma^i_\beta(V)$.
\end{claim}
\begin{proof}[Proof of Claim]
By the definition of $\rho$ and $\beta$, it is clear that $f^\star(T)\subset \Sigma^i_\beta(V)$. Let $S$ be the irreducible component of $\Sigma^i_\beta(V)$ containing $f^\star(T)$. We want to show that $S=f^\star(T)$. Let $\tilde{S}$ be a connected component of $(f^\star)^{-1}(S)$ containing $T$. Since $f^\star$ is a covering map, $\tilde{S}$ is irreducible and is a covering space of $S$. 

Suppose $T\subsetneq\tilde{S}$. Take a general point $\rho'$ in $\tilde{S}$. Since $\tilde{S}$ is irreducible, and since $T$ is an irreducible component of $\Sigma^i_k(U)$, we can assume $\rho'\notin \Sigma^i_k(U)$. Therefore, $\dim H^i(U, L_{\rho'})<\dim H^i(U, L_{\rho})$. Since $\rho'$ is more general than $\rho$, $\dim H^i(U, L_{\rho'}\otimes L_\tau^{\otimes j})\leq \dim H^i(U, L_{\rho}\otimes L_\tau^{\otimes j})$, for every $1\leq j\leq r-1$. Thus, 
$$
\sum_{j=0}^{r-1}\dim H^i(U, L_{\rho '}\otimes L_\tau^{\otimes j})<\sum_{j=0}^{r-1}\dim H^i(U, L_{\rho}\otimes L_\tau^{\otimes j}).
$$
Now, equality (\ref{eq}) implies that $\dim H^i(V, f^*(L_{\rho'}))<\beta$, and hence $f^\star(\rho')$ is not contained in $\Sigma^i_\beta(V)$. This is a contradiction to the definition of $\rho'$ and $\tilde{S}$. So we have proved $T=\tilde{S}$. Therefore, $f^\star(T)=S$ is an irreducible component of $\Sigma^i_\beta(V)$. 
\end{proof}
Since $\tau\in T$, $f^\star(T)$ contains $\mathbf{1}$. By the assumption of the theorem, $f^\star(T)$ is a subtorus in $\mb(V)$. Since $f^\star$ is a covering map, obviously $T$ must be a translate of subtorus. We finished the proof of the proposition. 
\end{proof}

Theorem \ref{torsion} is a direct consequence of Theorem \ref{torsionp}, Proposition \ref{cover}, and Theorem \ref{dimca}.

\section{Appendix}
We prove the following strengthening of Lemma \ref{cstar} pointed out to us by Jiu-Kang Yu.

\begin{lemma}
Suppose $S\subset \cc^n$ is a Zariski closed subset defined over $\bar\qq$. Suppose $T\subset (\cc^*)^n$ is also a Zariski closed subset defined over $\bar\qq$ such that $\dim S=\dim T$ and $\exp(S)\subset T$. Then $T$ is a torsion translate of a subtorus.
\end{lemma}
\begin{proof}
First we prove the lemma for the case $\codim(S)=1$. Denote the projections to the first $n-1$ coordinates by $p_1:\cc^n\to \cc^{n-1}$ and $p_2: (\cc^*)^n\to (\cc^*)^{n-1}$. After a change of bases, we can assume that $\dim(p_1(S))=\dim(p_2(T))=n-1$. 

Let $\rho\in \qq^{n-1}\subset\cc^{n-1}$ be a point with rational coordinates. Denote $p_1^{-1}(\rho)$ and $p_2^{-1}(\exp(\rho))$ by $A_\rho$ and $B_\rho$, respectively. Since $\dim(p_2(T))=n-1$, for a general $\rho$, $B_\rho\cap T$ consists of finitely many points. Since $\exp(S)\subset T$ and $\exp(A_\rho)=B_\rho$, we have 
$$\exp(A_\rho\cap S)\subset B_\rho \cap T.$$
The projection to the last coordinate defines an isomorphism $A_\rho\cong \cc$. Similarly we have $B_\rho\cong \cc^*$. Under these isomorphisms, $A_\rho\cap S$ and $B_\rho \cap T$ are both defined over $\bar\qq$. This means any point in $A_\rho\cap S$ is a $\bar\qq$ point, and its image under the exponential is also a $\bar\qq$ point. Now, according to Gelfond-Schneider theorem, the points in $A_\rho\cap S$ must be rational points. 

We have shown that for a general $\rho\in \cc^{n-1}$, $A_\rho\cap S$ consists of only points with rational coordinates. Suppose $S$ is defined by a polynomial $f(x_1, \ldots, x_n)=0$ with coefficients in $\bar\qq$. Since $S$ is irreducible, $f$ is irreducible over $\bar\qq$. Let $\bar{f}$ be the irreducible polynomial defined over $\qq$ that has $f$ as a factor over $\bar\qq$. Now, for a general $\rho\in \qq^{n-1}$, the intersection of the zero locus of $\bar{f}$ and $A_\rho$ must contain at least one point with rational coordinates. This means that by plugging in a general $(n-1)$-tuple of rational numbers into the first $n-1$ variables, $\bar{f}(x_1, \ldots, x_n)=0$ has at least one solution $x_n\in\qq$. However, by Hilbert irreducibility theorem, after plugging in such a general $(n-1)$-tuple of rational numbers, $\bar{f}$ is irreducible over $\qq$ as a polynomial in $x_n$. Therefore, $\bar{f}$ must be of degree one in $x_n$. Since the coordinates can be chosen generically, $\bar{f}$ itself is of degree one. Now, it is obvious that $S$ is a translate of a linear subspace defined over $\qq$, and $T$ is a translate of a subtorus by a torsion point. 

Next, we use induction on the codimension of $S$. Suppose $\codim(S)\geq 2$. We define the projections $p_1:\cc^n\to \cc^{n-1}$ and $p_2: (\cc^*)^n\to (\cc^*)^{n-1}$ as before. After a change of bases, we can assume that $\dim(p_1(S))=\dim(S)$. Then, $\dim(p_2(T))=\dim(T)=\dim(S)$. 

Let $S'=\overline{p_1(S)}$ and  $T'=\overline{p_2(T)}$ be the closures in the usual Euclidean topology. Since $p_1(S)$ and $p_2(T)$ are Zariski constructible sets, both the Zariski topology and the usual topology define the same closure. Hence $S'$ and $T'$ are defined over $\bar\qq$. Since the exponential map is continuous in the usual topology, 
$$\exp\left(\overline{p_1(S)}\right)\subset \overline{\exp(p_1(S))}.$$
Since $\exp(S)\subset T$ and $\exp(p_1(S))=p_2(\exp(S))$, we have
$$\exp(S')=\exp\left(\overline{p_1(S)}\right)\subset \overline{\exp(p_1(S))}=\overline{p_2(\exp(S))}\subset \overline{p_2(T)}=T'.$$
Using the induction hypothesis on the pair $S'\subset \cc^{n-1}$ and $T'\subset (\cc^*)^{n-1}$, we conclude that $T'$ is a torsion translate of a subtorus. Now, by choosing a torsion point of $T'$ as origin, we can identify $T'$ as $(\cc^*)^{\dim(T)}$. Taking the connected component of $\exp ^{-1}(p_2^{-1}(T'))$ containing $S$ and choosing a compatible origin on this connected component, the problem is reduced to a codimension one case, which is already solved. 
\end{proof}

\bigskip
\end{document}